\documentclass[10pt]{amsart}
\usepackage{graphicx}
\usepackage{amscd}
\usepackage{amsmath}
\usepackage{amsthm}
\usepackage{amsfonts}
\usepackage{amssymb}
\usepackage{mathrsfs}
\usepackage{bm}
\usepackage{enumerate}
\usepackage{amsrefs}
\usepackage{xcolor}
\usepackage[colorlinks, citecolor=blue, linkcolor=red, pdfstartview=FitB]{hyperref}

\usepackage{tikz}
\usepackage{tikz-cd}
\usepackage{caption}
\usetikzlibrary{decorations.markings}
\usepackage{lipsum}
\usepackage{mathrsfs}

\usepackage{marginnote}	
\usepackage{soul} 

\newcommand{\bB}{{\mathbb{B}}}
\newcommand{\bC}{{\mathbb{C}}}

\newcommand{\bE}{{\mathbb{E}}}
\newcommand{\bF}{{\mathbb{F}}}

\newcommand{\bM}{{\mathbb{M}}}
\newcommand{\bN}{{\mathbb{N}}}

\newcommand{\bT}{{\mathbb{T}}}

\newcommand{\bZ}{{\mathbb{Z}}}

  \newcommand{\A}{{\mathcal{A}}}
  \newcommand{\B}{{\mathcal{B}}}
  \newcommand{\C}{{\mathcal{C}}}
  \newcommand{\D}{{\mathcal{D}}}
  
  \newcommand{\F}{{\mathcal{F}}}
  
\renewcommand{\H}{{\mathcal{H}}}

  \newcommand{\K}{{\mathcal{K}}}  
\renewcommand{\L}{{\mathcal{L}}}

\renewcommand{\O}{{\mathcal{O}}}

  \newcommand{\T}{{\mathcal{T}}}

\newcommand{\fB}{{\mathfrak{B}}}

\newcommand{\rC}{\mathrm{C}}

\renewcommand{\phi}{\varphi}

\newcommand{\upchi}{{\raise.35ex\hbox{$\chi$}}}

\newcommand{\ol}{\overline}

\newcommand{\mycomment}[1]{}

\newcommand{\id}{\operatorname{id}}

\newcommand{\acts}{\curvearrowright}

\newtheorem{theorem}{Theorem}[section]
\newtheorem*{theorem*}{Theorem}
\newtheorem{lemma}[theorem]{Lemma}
\newtheorem*{lemma*}{Lemma}
\newtheorem{corollary}[theorem]{Corollary}

\newtheorem{proposition}[theorem]{Proposition}

\theoremstyle{definition}
\newtheorem{example}[theorem]{Example}

\theoremstyle{plain}
\newtheorem{theoremx}{Theorem}

\date{\today}

\author{Adam Dor-On}
\address{Department of Mathematics, University of Haifa, Mount Carmel, Haifa 3103301, Israel}
\email{adoron.math@gmail.com\vspace{-1ex}}

\author{Ian Thompson}
\address{Department of Mathematical Sciences, University of Copenhagen, Universitetsparken 5, 2100 Copenhagen, Denmark}
\email{ian@math.ku.dk\vspace{-1ex}}



\thanks{A. Dor-On was partially
supported by an NSF-BSF grant no. 2350543 / 2023695 (respectively), an NSF / BSF grant no. 2452324 / 2024734 (respectively), and a DFG Middle-Eastern collaboration project no. 529300231. I. Thompson was partially supported by an NSERC Postdoctoral Fellowship.}

\title[Obstructions to the full Hao--Ng isomorphism]{Obstructions to the full Hao-Ng isomorphism}
\begin{document}
\begin{abstract}
    We show that the Hao--Ng isomorphism for full crossed products fails to hold. More precisely, for a non-degenerate $\rC^*$-correspondence $X$ over a $\rC^*$-algebra $\A$ and a generalized gauge action $G \curvearrowright X$ by a locally compact group $G$, we isolate three structural obstructions for the canonical commutation $\O_X\rtimes G \cong \O_{X\rtimes G}$ between full crossed product and Cuntz--Pimsner $\rC^*$-algebra. The first concerns hyperrigidity, the second concerns WEP of the coefficient $\rC^*$-algebra $\A$ for trivial actions, and the third concerns the coincidence of full and reduced crossed products by the action of $G$ on $\A$.
\end{abstract}
\maketitle

\section{Introduction}\label{s:intro}

Dynamical systems of $\rC^*$-algebras are a cornerstone to modern operator algebra theory, readily connecting to topological dynamics, mathematical physics, and classification of $\rC^*$-algebras. An important class of $\rC^*$-dynamical systems are quasi-free actions on Cuntz algebras, which have seen intense focus \cites{doplicher1989endomorphisms, doplicher1989new, izumi2004finite, izumi2004finite2} as they are completely determined on the generating isometries. This symbiosis between $\rC^*$-algebra and operator space dynamics has been of an unusual benefit for the purposes of classification of these actions up to cocycle conjugacy, and it begs the question of how to interpret such interactions among broader classes of dynamics on operator algebras.

In this work, we show that the full Hao--Ng isomorphism problem has a negative answer in an especially strong sense. The problem asks when the Cuntz--Pimsner $\rC^*$-algebra $\O_X$ construction commutes with the full crossed product $\rC^*$-algebra. To this end, an action $\alpha:G\acts\O_X$ by a locally compact group $G$ is said to be a generalized gauge action if it naturally arises from an action on the generating $\rC^*$-correspondence $(\A, X)$. This constraint gives rise to a corresponding full crossed product $\rC^*$-correspondence $(\A\rtimes_\alpha G, X\rtimes_\alpha G)$. The full Hao--Ng isomorphism then asks if the Cuntz--Pimsner algebra $\O_{X\rtimes_\alpha G}$ is the full crossed product of the Cuntz--Pimsner algebra $\O_X \rtimes_\alpha G$.

The Hao--Ng isomorphism problem for amenable groups was resolved by Hao and Ng around 18 years ago \cite{hao2008crossed}. This isomorphism problem has since been closely interwoven with functoriality and crossed product duality \cites{kaliszewski2013functoriality, kaliszewski2015coactions, bedos2015new}, and has seen applications in various contexts \cites{deaconu2012group, deaconu2018group, schafhauser2015cuntz}. Furthermore, non-self-adjoint techniques for establishing the \emph{reduced} Hao--Ng isomorphism have seen consistent improvements and extensions over the last decade \cites{katsoulis2017c, katsoulis2021non, kakariadis2025fock, brix2026normal, paraskevas2026reduced}, with the first solution accomplished by the authors in the past year \cite{dor2025hao} (see \cite[Section 6]{mukohara2026stabilization} as well, for a $\rC^*$-algebraic proof). Nevertheless, the Hao--Ng isomorphism for full crossed products has not witnessed the same level of progress. B\'edos, Kaliszewski, Quigg, and Spielberg verified the isomorphism for row-finite graph correspondences by discrete groups \cite[Corollary 6.8 and Remark 6.10]{bedos2018finitely}, and Katsoulis and Ramsey extended their result to general locally compact groups \cite{katsoulis2021non}. However, beyond reformulations of the problem  \cites{bedos2015new,katsoulis2019crossed,katsoulis2021non}, the full Hao--Ng isomorphism has remained largely mysterious.

Here, we present the first counterexamples to the Hao--Ng isomorphism for full crossed products, showing distinct obstructions to the validity of this isomorphism. As in \cite{dor2025hao}, our work is inspired by the strategy of establishing Hao--Ng isomorphisms through the more general theory of non-self-adjoint operator algebras and the $\rC^*$-envelope, which emerged from the work of Katsoulis and Ramsey \cites{katsoulis2017c, katsoulis2019crossed, katsoulis2021non}. Katsoulis and Ramsey's program was concerned with the general development of non-self-adjoint dynamical systems, with a central goal of determining the commutation of the $\rC^*$-envelope with either the full or reduced crossed products, and thereby resolving the Hao--Ng isomorphism problems in full and reduced crossed products. This led them to prove that the full Hao--Ng isomorphism is equivalent to the commutation between the $\rC^*$-envelope and full crossed product for the corresponding \emph{tensor algebra} dynamical system \cite[Theorem 4.9]{katsoulis2021non}, which is a particular subalgebra of the Cuntz--Pimsner $\rC^*$-algebra.

However, in \cite[Theorem 5.6]{harris2019crossed} Harris and Kim constructed a non-self-adjoint dynamical system, which is not a tensor algebra dynamical system, for which the full crossed product does not commute with the $\rC^*$-envelope. Full crossed products of non-self-adjoint dynamical systems are necessarily constructed from the full crossed product of the so-called maximal $\rC^*$-cover of the non-self-adjoint operator algebra. Unfortunately, the structure of the maximal $\rC^*$-cover can be very mysterious, even in well-studied examples \cite[Example 2.3]{blecher1999modules}. Nevertheless, as a byproduct of \cite[Theorem 4.9]{katsoulis2021non}, it was shown that the full crossed product of a tensor algebra must coincide with the corresponding construction arising from the full crossed product of the $\rC^*$-envelope for the full Hao--Ng isomorphism to be valid. For \emph{hyperrigid} $\rC^*$-correspondences, which are those $\rC^*$-correspondences $(\A, X)$ where Katsura's ideal acts non-degenerately on $X$ \cite{kim2021hyperrigidity}, the converse was also shown to hold.

The first counterexample we present to the full Hao--Ng isomorphism concerns $\rC^*$-correspondences that fail to be hyperrigid. In Theorem \ref{t:infinite-cuntz}, through a concrete construction, we show that the absence of hyperrigidity can invoke a strong failure of the Hao--Ng isomorphism for full crossed products.

\begin{theoremx}\label{tx:O-inf}
    Let $\A = \mathbb{C}$ and $X = \ell^2(\mathbb{N})$, so that $\mathcal{O}_X = \mathcal{O}_{\infty}$. For every countable non-amenable discrete group $G$, there is a generalized gauge action $\alpha:G\acts(\A, X)$ so that the induced generator-preserving surjective $*$-homomorphism $\O_X\rtimes_\alpha G \rightarrow \O_{X\rtimes_\alpha G}$ fails to be injective.
\end{theoremx}

Accordingly, we turn our attention towards the question of the full Hao--Ng isomorphism for hyperrigid $\rC^*$-correspondences. We uncover a second obstruction for trivial actions by free groups in Theorem \ref{t:coeff-obs}. This confirms the suspicion of a few works \cites{bedos2015new, katsoulis2019crossed}, showing that the full Hao--Ng isomorphism for trivial actions is connected with the weak expectation property. For this, we require a characterization of nuclearity-related properties for tensor algebras. We then construct dynamical systems of hyperrigid $\rC^*$-correspondences that witness this obstruction in Example \ref{e:tens}, thereby refuting the full Hao--Ng isomorphism for trivial actions.

\begin{theoremx}\label{tx:hr-triv}
    Let $X$ be a hyperrigid $\rC^*$-correspondence over the $\rC^*$-algebra $\A$, so that $\O_X$ has the WEP, yet $\A$ does not have WEP. Let $\bF$ be a discrete non-abelian free group. If $\id$ denotes the trivial action of $\bF$ on $(\A, X)$, the induced generator-preserving surjective $*$-homomorphism $ \O_{X\rtimes_{\id} \bF} \rightarrow\O_X\rtimes_{\id}\bF$ fails to be injective.
\end{theoremx}

In Theorem \ref{t:act-obs}, we then uncover a third obstruction to the full Hao-Ng isomorphism for hyperrigid $\rC^*$-correspondences by non-trivial actions. A family of regular topological graphs that meets this criterion is then constructed in Example \ref{e:top-graph}.

\begin{theoremx}\label{tx:hr-obs}
    Let $X$ be a hyperrigid $\rC^*$-correspondence over the $\rC^*$-algebra $\A$, and let $\alpha : G \acts (\A,X)$ be a generalized gauge action. If the canonical map $\O_X\rtimes_\alpha G \rightarrow \O_X\rtimes_{\alpha, r} G$ is injective yet $\A\rtimes_\alpha G \rightarrow \A \rtimes_{\alpha,r} G$ is not, then the induced generator-preserving surjective $*$-homomorphism $\O_{X\rtimes_\alpha G}\longrightarrow \O_X\rtimes_\alpha G$ fails to be injective.
\end{theoremx}

We now present the structure of this paper. In Section \ref{s:prelim}, we give background on non-self-adjoint operator algebra theory, as well as algebras formed from $\rC^*$-correspondences, and dynamics. In Section \ref{s:nuc-prop}, we provide some useful characterizations of nuclearity, exactness and the weak expectation property for tensor algebras. In Section \ref{s:Hao--Ng}, we present our obstructions to the Hao--Ng isomorphism for full crossed products. In Subsection \ref{ss:O-inf}, we show the failure of the full Hao--Ng isomorphism for non-hyperrigid $\rC^*$-correspondences. In Subsection \ref{ss:hr-Hao--Ng}, we establish counterexamples to the full Hao--Ng isomorphism for hyperrigid $\rC^*$-correspondences, together with nuclearity-related properties for tensor algebras.\\

\textbf{Acknowledgments.} The first-named author is grateful to Kevin Aguyar Brix and Chris Bruce for discussions on nuclearity and exactness of $\rC^*$-algebras held at an ICMS meeting in June 2025. The authors thank Elias Katsoulis for informing them that Theorem \ref{tx:O-inf} resolves questions he had considered. 

We announced parts of our work on the full Hao-Ng isomorphism problem in November 2025 (Michigan State University \& Technion), March 2026 (Combinatorial $*$-algebras, MATRIX, Australia), and May 2026 (University of Potsdam). We are grateful to the organizers of these seminars and conferences for the opportunity to present this work. On June 12th, 2026, we contacted Miho Mukohara and Yuhei Suzuki regarding the first version of their pre-print \cite{mukohara2026stabilization} to inform them that we believe the full Hao-Ng isomorphism is false. In the third version of their pre-print \cite{mukohara2026stabilization}, they announced other counterexamples to the full Hao-Ng isomorphism.\\

\textbf{Statement on the use of AI.} ChatGPT 5.5 and Claude Opus 5 were used to help make final edits to this paper. ChatGPT 5.5 was also used to perform preliminary computations, specifically for testing the non-vanishing of Katsura’s ideal for families of crossed product correspondences of $\ell^2(\mathbb{N})$ as a $\mathbb{C}$-correspondence. The authors take full responsibility for the correctness and validity of this work.

\section{Preliminaries}\label{s:prelim}

\subsection{Operator algebras and their C*-algebras} \label{ss:C*-env}

An \emph{operator algebra} is a subalgebra of Hilbert space operators $\A\subset\bB(\K)$. An operator algebra can also be defined abstractly through a set of axioms \cite[Section 16]{paulsen2002completely}. Throughout, we assume all operator algebras are norm-closed and \emph{approximately unital}, in the sense that they possess a contractive approximate unit.

Given operator algebras $\A$ and $\B$, we let $\A\otimes\B$ denote the minimal (i.e.~ spatial) tensor product. The maximal tensor product $\A\otimes_{\max}\B$ is defined by the following universal property: there are completely isometric homomorphisms $\A,\B\hookrightarrow\A\otimes_{\max}\B$ such that whenever $\pi:\A\rightarrow\bB(\H)$ and $\rho:\B\rightarrow\bB(\H)$ are (completely contractive) representations with commuting ranges, there is a unique representation $\pi\otimes_{\max}\rho:\A\otimes_{\max}\B\rightarrow\bB(\H)$ that extends both $\pi$ and $\rho$. We remark that the maximal tensor product of operator algebras does not agree with the maximal tensor product for operator spaces. For more details on tensor products of non-self-adjoint operator algebras, we refer the reader to \cite{paulsen1990tensor},\cite[Sections 1.5 and 6.1]{blecher2004operator}.

We also recall some of Arveson's noncommutative boundary theory \cite{arveson1969subalgebras}. A \emph{$\rC^*$-cover} of $\A$ is a pair $(\fB, \iota)$ consisting of a $\rC^*$-algebra $\fB$ and a completely isometric homomorphism $\iota:\A\rightarrow \fB$ where $\fB = \rC^*(\iota(\A))$. Two canonical $\rC^*$-covers of $\A$ are the \emph{maximal $\rC^*$-cover}, denoted $(\rC^*_{\max}(\A), \mu)$, and the \emph{$\rC^*$-envelope}, denoted $(\rC^*_e(\A), \varepsilon)$. These $\rC^*$-covers have the defining properties that, whenever $(\fB, \iota)$ is a $\rC^*$-cover for $\A$, there are surjective $*$-homomorphisms that make the following diagram commute:
\[\begin{tikzcd}
    & \rC^*_{\max}(\A) \arrow[d, dashed]\\
    \A \arrow[ru, "\mu"] \arrow[r, "\iota"] \arrow[rd, "\varepsilon"] & \fB \arrow[d, dashed]\\
    & \rC^*_e(\A)
\end{tikzcd}\]
For any operator algebra, both of these $\rC^*$-covers are known to always exist \cite{blecher1999modules},\cite{hamana1979injective}.

An operator algebra $\A$ is said to be \emph{hyperrigid} if, whenever $\pi:\rC^*_e(\A)\rightarrow \bB(\H)$ is a faithful non-degenerate $*$-representation and $\psi:\rC^*_e(\A) \rightarrow \bB(\H)$ is a completely contractive completely positive map such that $\psi\circ\varepsilon = \pi\circ\varepsilon$, we have that $\psi = \pi$ \cite{arveson2011noncommutative}. Many of the most well-studied classes of operator algebras have now been verified to be hyperrigid, as can be seen for example in \cites{arveson2011noncommutative,  bilich2025maximality, clouatre2018multiplier, dor2018full, kim2021hyperrigidity}.

\subsection{Non-self-adjoint dynamical systems}\label{ss:nsa-dyn}

Let $\A$ be an operator algebra, $G$ be a locally compact Hausdorff group, and $\alpha:G\acts\A$ be a point-norm continuous action by completely isometric isomorphisms. The triple $(\A, G, \alpha)$ will be called a \emph{dynamical system}. If $\A$ is a $\rC^*$-algebra, we call the triple a \emph{$\rC^*$-dynamical system}.

A $\rC^*$-cover $(\fB, \iota)$ of $\A$ is said to be \emph{$\alpha$-admissible} when there is a continuous action $\widetilde{\alpha}: G\acts\fB$ such that $\widetilde{\alpha} \circ\iota = \iota\circ\alpha$.  When such an action $\widetilde{\alpha}$ exists, it is necessarily unique and so, we simply denote $\widetilde{\alpha}$ by $\alpha$ when the context is clear. Not every $\rC^*$-cover is $\alpha$-admissible \cite[Proposition 2.1]{katsoulis2021non}, but both the maximal $\rC^*$-cover and the $\rC^*$-envelope are $\alpha$-admissible \cite[Lemma 3.4]{katsoulis2019crossed}. Given an $\alpha$-admissible $\rC^*$-cover $(\fB, \iota)$ of $\A$, we define the \emph{full crossed product relative to $(\fB ,\iota)$} to be the norm-closed subalgebra $\A\rtimes_{\fB, \iota,\alpha}G$ of $\fB\rtimes_\alpha G$ that is generated by the canonical copy of $C_c(G;\iota(\A))$. For any $\alpha$-admissible $\rC^*$-cover $(\fB, \iota)$, there are completely contractive homomorphisms with dense ranges
\[
    \A\rtimes_{\rC^*_{\max}(\A), \mu, \alpha} G \longrightarrow \A\rtimes_{\fB, \iota, \alpha} G \longrightarrow \A\rtimes_{\rC^*_e(\A), \varepsilon, \alpha} G
\]
that fix the canonical copies of $\rC_c(G; \A)$ \cite[Theorem 2.4]{katsoulis2021non}. The operator algebra $\A\rtimes_{\fB, \iota,\alpha}G$ is generally dependent on the $\rC^*$-cover $(\fB, \iota)$, with the first such class of examples being discovered in the work of Harris and Kim \cite[Theorem 5.6]{harris2019crossed}. More precisely, there is an operator algebra $\A\subset \bM_r(\mathbb{C})$ such that the following property holds: whenever $G$ is a discrete non-amenable group, the quotient mapping $\A\rtimes_{\rC^*_{\max}(\A), \mu, \id} G \rightarrow \A\rtimes_{\rC^*_e(\A), \varepsilon, \id} G$ is not completely isometric.

The full crossed product relative to the maximal $\rC^*$-cover has an appropriate universal property describing a full crossed product structure for the dynamical system $(\A, G, \alpha)$. Explicitly, a pair $(\pi, u):(\A, G)\rightarrow\bB(\H)$ is said to be \emph{covariant} if $\pi:\A\rightarrow\bB(\H)$ is a completely contractive representation and $u:G\rightarrow\bB(\H)$ is a strongly continuous unitary representation with the property that
\[
    \pi(\alpha_g(a)) = u_g \pi(a)u_g^*, \ \text{for} \ \ g\in G, a\in\A.
\]
The operator algebra $\A\rtimes_{\rC^*_{\max}(\A), \mu, \alpha} G$ has the universal property that any covariant pair $(\pi,u):(\A, G)\rightarrow\bB(\H)$ extends to a completely contractive representation $\pi\rtimes u:\A\rtimes_{\rC^*_{\max}(\A), \mu, \alpha} G\rightarrow\bB(\H)$  \cite[Proposition 3.8]{katsoulis2019crossed}. Accordingly, the algebra
\[
    \A\rtimes_\alpha G := \A\rtimes_{\rC^*_{\max}(\A), \mu, \alpha} G
\]
is defined to be the \emph{full crossed product} of the dynamical system $(\A, G, \alpha)$.

The reduced crossed product is defined in an analogous way. However, contrary to the example of Harris and Kim, the reduced crossed product is independent of the choice of $\alpha$-admissible $\rC^*$-cover \cite[Corollary 3.16]{katsoulis2019crossed}. To describe the reduced crossed product, let $\lambda:G\rightarrow \bB( L^2(G))$ denote the left regular representation and, given a non-degenerate representation $\pi:\A\rightarrow \bB(\H)$, we define $\pi_\alpha: \A\rightarrow L^\infty(G;\bB(\H))$ by $\pi_\alpha(a)(g) = (\pi\circ\alpha_g^{-1})(a)$ for every $g\in G$ and $a\in\A$. The pair $(\pi_\alpha, \id\otimes\lambda)$ is readily seen to be covariant. One may then define the integrated form $\pi_\alpha\rtimes(\id\otimes\lambda)$ on $\rC_c(G;\A)$ by $[\pi_\alpha \rtimes(\id\otimes\lambda)](f) = \int_G \pi_\alpha(f(g))\id\otimes\lambda_g dm(g)$ \cite[Lemma 1.91]{williams2007book}. Now, given any $\alpha$-admissible $\rC^*$-cover $(\fB, \iota)$, the \emph{reduced crossed product} is defined as the closure of the image of $\rC_c(G;\A)$ under $\iota_\alpha \rtimes(\id\otimes\lambda)$. For more details on non-self-adjoint dynamical systems, we refer the reader to \cite{katsoulis2019crossed}.

\subsection{Cuntz--Pimsner C*-algebras and crossed product correspondences}\label{ss:correspondence}

A \emph{$\rC^*$-correspondence} is a pair $(\A, X)$ consisting of a $\rC^*$-algebra $\A$ and a right Hilbert $\A$-module $X$, together with a left action $\varphi_X:\A\rightarrow\L(X)$ by adjointable operators on $X$. Throughout, we assume that $(\A, X)$ is \emph{non-degenerate} in the sense that $X = \ol{\varphi_X(\A)X}$. Given $x,y\in X$, one can define an adjointable operator on $X$ by $\theta_{x,y}(z) = x\langle y, z\rangle$. The norm-closure of the linear span of $\{\theta_{x,y} : x,y\in X\}$ will then be denoted by $\K(X)$, and is an ideal in $\L(X)$ referred to as the generalized compact operators on $X$. For details on the constructions found in this section, one can consult \cite[Section 4.6]{brown2008textrm} and \cite{katsura2004c}.

For a pair of $\rC^*$-correspondences $X,Y$ over $\A$, there is a natural notion of a tensor product $\rC^*$-correspondence $ X\otimes Y$ over $\A$. Consequently, a $\rC^*$-correspondence $(\A, X)$ gives rise to a \emph{Fock correspondence} $\F_X$ over $\A$, which is defined by
\[
    \F_X = \A \oplus \bigoplus_{n=1}^\infty X^{\otimes n}
\]
with $X^{\otimes n}$ denoting the $n$-fold tensor product of $X$. Attached to the Fock correspondence $\F_X$ is a family of left-creation operators $\{L_x\in\L(\F_X) : x\in X\}$ where
\[
    L_x(a) = xa, \quad L_x(x_1\otimes\ldots\otimes x_n) = x\otimes x_1\otimes\ldots \otimes x_n
\]
whenever $a\in\A$ and $x,x_1, \ldots, x_n\in X$. The \emph{Toeplitz $\rC^*$-algebra} $\T_X$ of the $\rC^*$-correspondence $(\A, X)$ is the $\rC^*$-subalgebra of $\L(\F_X)$ generated by $\varphi_{\F_X}(\A)$, which is identified with a copy of $\A$, and the left-creation operators $\{L_x\in\L(\F_X) : x\in X\}$, which is identified with a copy of $X$. The corresponding norm-closed operator algebra generated by $\varphi_{\F_X}(\A)$ and the left-creation operators is called the \emph{tensor algebra} of $X$, which will be denoted by $\T_X^+$. As $(\A, X)$ is assumed to be non-degenerate, the tensor algebra admits a contractive approximate unit from $\A$.

We recall the construction of the Cuntz-Pimsner $\rC^*$-algebra built from $(\A, X)$. To this end, we consider the closed two-sided ideal
\[
    J_X = \{a\in \A ~|~ \varphi_X(a)\in\K(X) \quad \text{and} \quad ab = 0 \ \text{for all} \ b\in\ker\varphi_X\},
\]
which is referred to as \emph{Katsura's ideal}. Given another $\rC^*$-algebra $\B$, a \emph{representation} of $(\A, X)$ is a pair $(\rho, t)$ such that $t:X\rightarrow \B$ is a completely contractive linear map and $\rho:\A\rightarrow \B$ is a (non-degenerate) $*$-homomorphism satisfying
\[
    \rho(a_1)t(x)\rho(a_2) = t(a_1 x a_2), \quad a_1, a_2\in \A, x\in X.
\]
Here, we remark that we identify $\varphi_X:\A\rightarrow\L(X)$ with an action of $\A$ on $X$. When $t(x)^*t(y) = \rho(\langle x, y\rangle)$ for every $x,y\in X$, then we say that $(\rho, t)$ is \emph{rigged} (or \emph{isometric}, as in \cite{muhly1998tensor}). Each rigged representation gives rise to a $*$-homomorphism $\psi_t:\K(X)\rightarrow\B$ where $\psi_t(\theta_{x,y}) = t(x)t(y)^*$ for every $x,y\in X$. If $\psi_t(\varphi_X(a)) = \rho(a)$ for each member $a\in J_X$ of Katsura's ideal, then $(\rho, t)$ is said to be \emph{covariant}. The universal $\rC^*$-algebra generated by rigged representations of $(\A,X)$ is known to coincide with the Toeplitz algebra $\T_X$ \cites{muhly1998algebraic, pimsner1996class}, whereas the universal $\rC^*$-algebra generated by the rigged covariant representations is the so-called \emph{Cuntz--Pimsner algebra}, denoted by $\O_X$. In particular, the Cuntz--Pimsner $\rC^*$-algebra $\O_X$ is the quotient of $\T_X$ by the closed two-sided ideal generated by $\{\psi_t(\varphi_X(a)) - a : a\in J_X\}$.

Let $G$ be a locally compact Hausdorff group. A \emph{generalized gauge action} is a point-norm continuous action $\alpha:G\acts\T_X$ by $*$-automorphisms such that $\alpha_g(\A) = \A$ and $\alpha_g(X) = X$ for each $g\in G$. We will refer to the triple $((\A, X), G, \alpha)$ as a \emph{$\rC^*$-correspondence dynamical system}. A $\rC^*$-correspondence dynamical system gives rise to full and reduced crossed product $\rC^*$-correspondences (see for instance \cite[Section 2]{bedos2015new}). In this paper, we focus on the full crossed product of a $\rC^*$-correspondence dynamical system. Explicitly, the \emph{full crossed product $\rC^*$-correspondence} is the $\rC^*$-correspondence $(\A\rtimes_\alpha G, X\rtimes_\alpha G)$ obtained from the completion of $\rC_c(G; \A)$ and $\rC_c(G; X)$ in $\T_X\rtimes_\alpha G$. Here, the bimodule action is given by multiplication and the Hilbert module structure is defined by $\langle g,h\rangle = g^*h$ for $g,h\in\rC_c(G; X)$. Given a $\rC^*$-correspondence dynamical system $((\A, X), G, \alpha)$ and a pair of representations $(\rho_1, t_1)$ and $(\rho_2, t_2)$ for $(\A\rtimes_\alpha G, X \rtimes_\alpha G)$, we say a $*$-homomorphism between $\rC^*(\rho_1, t_1)$ and $\rC^*(\rho_2, t_2)$ is generator-preserving if it is the identity on the canonical copies of $\rC_c(G;\A)$ and $\rC_c(G; X)$.

In addition, we define the $\rC^*$-correspondence $(\A\hat{\rtimes}_\alpha G, X\hat{\rtimes}_\alpha G)$ as the completion of $\rC_c(G; \A)$ and $\rC_c(G; X)$ in $\O_X\rtimes_\alpha G$. We will see that in general, it can be the case that the canonical quotient $\A \rtimes_{\alpha} G \rightarrow \A \hat{\rtimes}_{\alpha} G$ fails to be injective. Similarly, the \emph{reduced crossed product $\rC^*$-correspondence} $(\A\rtimes_{\alpha, r}G, X\rtimes_{\alpha,r}G)$ is the completion of $\rC_c(G;\A)$ and $\rC_c(G; X)$ in $\T_X \rtimes_{\alpha, r}G$, with bimodule action as multiplication and by setting $\langle g,h\rangle = g^*h$ for $g,h\in\rC_c(G; X)$.

A central fact we require is that the $\rC^*$-envelope of the tensor algebra $\T_X^+$ is the Cuntz--Pimsner algebra $\O_X$ \cite{katsoulis2006tensor}. Here, we pay particular attention to the fact that the quotient mapping $\T_X\rightarrow \O_X$, which exists by universality, is completely isometric on $\T_X^+$. If $\T_X^+$ is hyperrigid, we refer to the $\rC^*$-correspondence $(\A, X)$ as \emph{hyperrigid}. The work of Kim \cite{kim2021hyperrigidity} shows that a $\rC^*$-correspondence $(\A, X)$ is hyperrigid if and only if Katsura's ideal $J_X$ acts non-degenerately on $X$, i.e.~ $\ol{\varphi_X(J_X)X} = X$. Finally, a $\rC^*$-correspondence $(\A, X)$ is said to be \emph{regular} if the left action is injective and maps into $\K(X)$. In this case, Katsura's ideal $J_X$ coincides with $\A$ and $(\A, X)$ is hyperrigid by \cite{kim2021hyperrigidity}.

\section{Nuclearity properties for tensor algebras} \label{s:nuc-prop}

In this section, we characterize Blecher and Duncan's non-self-adjoint nuclearity-related properties \cites{blecher2011nuclearity, blecher2011errata} within the context of tensor algebras of $\rC^*$-correspondences. Specifically, we characterize notions of exactness, the weak expectation property, and nuclearity for tensor algebras. For this, let $(\A, X)$ be a $\rC^*$-correspondence and recall that nuclearity (respectively, exactness) of the Toeplitz $\rC^*$-algebra $\T_X$ is equivalent to the nuclearity (exactness) of the coefficient $\rC^*$-algebra $\A$ \cite[Theorems 7.1 and 7.2]{katsura2004c}. We will obtain similar characterizations for the tensor algebra of a $\rC^*$-correspondence.

For this, we say that an operator algebra $\A$ is \emph{exact} if, whenever  
\[
    0\longrightarrow \B \longrightarrow \C \longrightarrow \D\longrightarrow0
\]
is an exact sequence of operator algebras where the quotient mapping is a complete quotient map and the embedding is completely isometric, the same is true of
\[
    0\longrightarrow \A\otimes\B \longrightarrow \A\otimes\C \longrightarrow \A\otimes\D\longrightarrow0.
\]
In \cite{blecher2011nuclearity}, this property was referred to as being OA-exact and it was shown that this coincides with the familiar notion of exactness for operator spaces \cite[Proposition 5.2 (4)]{blecher2011nuclearity}. However, in contrast to the operator space setting (see \cite[Proposition 6.5]{blecher2011nuclearity}), exactness of non-self-adjoint operator algebras is preserved under subalgebras \cite[Proposition 5.2]{blecher2011nuclearity}. Accordingly, we have an immediate characterization of exactness for tensor algebras.

\begin{proposition}\label{p:tens-exact}
    Let $(\A, X)$ be a $\rC^*$-correspondence. Then $\T_X^+$ is exact if and only if $\A$ is exact.
\end{proposition}

\begin{proof}
    This follows from noting that $\A\subset\T_X^+\subset\T_X$ and that the Toeplitz algebra $\T_X$ is exact if and only if $\A$ is exact \cite[Theorem 7.1]{katsura2004c}.
\end{proof}

Recall that there are canonical conditional expectations $\bE:\T_X\rightarrow\A$ and $\bE:\T_X^+\rightarrow \A$ onto $\A$ \cite[Theorem 4.6.6 (2)]{brown2008textrm}. For the next results, we will need to gather a few facts about tensor products of operator algebras arising from $\rC^*$-correspondences.

\begin{lemma}\label{l:tens-max-incl}
    Let $(\A, X)$ be a $\rC^*$-correspondence and $\B$ be an arbitrary $\rC^*$-algebra. Then, the following statements hold.\begin{enumerate}[{\rm (i)}]
        \item The canonical inclusion $\A\otimes_{\max} \B\subset \T_X^+\otimes_{\max} \B$ is completely isometric.
        \item If there is a completely isometric isomorphism $\T_X^+\otimes_{\max} \B\cong\T_X^+\otimes \B$, then $\A\otimes_{\max} \B\cong \A\otimes \B$.
        \item We have completely isometric isomorphisms $\T_{X\otimes_{\max} \B}^+\cong \T_X^+ \otimes_{\max} \B$ and $\T_{X\otimes_{\max} \B}\cong \T_X \otimes_{\max} \B$ that are generator-preserving.
        \item The canonical inclusion $\T_X^+\otimes_{\max}\B\subset\T_X\otimes_{\max}\B$ is completely isometric.
    \end{enumerate}
\end{lemma}

\begin{proof}
    (i): By the universal property of the maximal tensor product, there is a completely contractive homomorphism
    \[
        \lambda: \A\otimes_{\max}\B\rightarrow \T_X^+\otimes_{\max} \B
    \]
    that extends $\id\odot\id$. We verify that $\lambda$ is injective. To this end, let $\pi: \A\otimes_{\max}\B\rightarrow \bB(\H)$ be a faithful $*$-homomorphism, and let 
    \[
        \pi_\A:\A\rightarrow\bB(\H) \quad \text{and} \quad  \pi_\B:\B\rightarrow\bB(\H)
    \]
    denote the $*$-homomorphisms arising from the restriction maps. Now, define the completely contractive homomorphism
    \[
        \rho:=\pi_\A\circ\bE: \T_X^+\rightarrow \bB(\H)
    \]
    where $\bE:\T_X^+\rightarrow \A$ denotes the canonical conditional homomorphism onto the coefficient $\rC^*$-algebra $\A$. The universal property of $\T_X^+\otimes_{\max} \B$ allows us to define a completely contractive homomorphism
    \[
        \rho\otimes_{\max}\pi_\B: \T_X^+\otimes_{\max} \B\rightarrow \pi_\A(\A)\otimes_{\max}\pi_\B(\B)
    \]
    that extends $\rho\odot\pi_\B$. Thus, we have a commuting diagram
    \[\begin{tikzcd}
        \T_X^+\otimes_{\max} \B \arrow[r, "\rho\otimes_{\max}\pi_\B"] &[2em] \pi_\A(\A)\otimes_{\max} \pi_\B(\B) \arrow[d]\\
        \A\otimes_{\max} \B \arrow[u, "\lambda"] \arrow[r, "\pi"] & \bB(\H)
    \end{tikzcd}\]
    where $\pi$ is injective. By commutativity of the diagram, $\lambda$ is then completely isometric. In other words, the inclusion $\A\otimes_{\max} \B\subset \T_X^+\otimes_{\max} \B$ is completely isometric.

    (ii): Continuing from (i), for every $x\in \A\odot \B$, we now have that
    \[
        \|x\|_{\max} = \|\lambda(x)\|_{\T_X^+ \otimes_{\max}\B} = \|\lambda(x)\|_{ \T_X^+\otimes\B} = \| x\|_{\A\otimes\B} = \|x\|_{\min}.
    \]
    Therefore, the canonical map $\A\otimes_{\max}\B\rightarrow \A\otimes\B$ is isometric, as desired.
    
    {\rm (iii)}: We prove that $\T_X^+\otimes_{\max} \B$ and $\T_{X\otimes_{\max}\B}^+$ share the same universal property \cite{muhly1998tensor}. First note that if $\rho:\T_X^+\otimes_{\max} \B\rightarrow\bB(\H)$ is a completely contractive representation, then
    \[
        (\rho|_{\A\otimes_{\max} \B},\rho|_{X\otimes_{\max} \B}, \H)
    \]
    is a completely contractive representation of $(\A\otimes_{\max} \B, X\otimes_{\max} \B)$ by restriction and the application of (i). Conversely, suppose that $(\chi, t, \H)$ is a completely contractive representation of $(\A\otimes_{\max}\B, X\otimes_{\max}\B)$. Letting $\iota_\A: \A\rightarrow \A\otimes_{\max} \B$ and $\iota_X: X\rightarrow X\otimes_{\max} \B$ be the canonical completely isometric inclusions, we then have a corresponding representation $(\chi\circ\iota_\A, t\circ\iota_X, \H)$ of $(\A, X)$. So, there is a completely contractive representation $\pi:\T_X^+\rightarrow\bB(\H)$ that is induced from $(\chi\circ\iota_\A, t\circ\iota_X, \H)$. Moreover, the range of $\pi$ commutes with the range of $\pi':=\chi\circ\iota_\B$, where $\iota_\B:\B\rightarrow \A\otimes_{\max}\B$ denotes the other canonical inclusion. It then follows that there is a unique completely contractive representation $\pi\otimes_{\max}\pi'$ of $\T_X^+\otimes_{\max}\B$, and it must agree with $(\chi, t, \H)$ on $(\A\otimes_{\max}\B, X\otimes_{\max}\B)$. Thus, by coincidence of universal properties, we conclude that $\T_X^+\otimes_{\max}\B\cong\T_{X\otimes_{\max}\B}^+$. For the Toeplitz $\rC^*$-algebra analogue of this identification, we argue the same way by using rigged representations in place of completely contractive ones.
  
    {\rm (iv)}: This is immediate from the generator-preserving isomorphisms in {\rm (iii)} since
    \[
        \T_X^+\otimes_{\max} \B\cong\T_{X\otimes_{\max}\B}^+\subset \T_{X\otimes_{\max}\B}\cong\T_X\otimes_{\max}\B.
    \]
\end{proof}

We now turn our attention to the weak expectation property for operator algebras arising from $\rC^*$-correspondences. A $\rC^*$-algebra $\B$ is said to have the \emph{weak expectation property} (WEP) if $\B\otimes_{\max}\rC^*(\bF)\cong\B\otimes\rC^*(\bF)$ for every (equivalently, some) discrete free group $\bF$. Due to work of Kirchberg \cite[Proposition 1.1]{kirchberg1993non}, this is equivalent to the usual notion that concerns the existence of weak expectations \cite[Definition 3.6.7]{brown2008textrm}. We remark as well that a $\rC^*$-algebra $\B$ is nuclear precisely when it is exact and has the weak expectation property. We summarize the connection between tensor algebras and the weak expectation property as follows.

\begin{proposition}\label{p:tens-wep}
    Let $(\A, X)$ be a $\rC^*$-correspondence. Then, the following statements are equivalent.\begin{enumerate}[{\rm (i)}]
        \item $\A$ has the WEP.
        \item $\T_X$ has the WEP.
        \item We have $\T_X^+\otimes_{\max}\rC^*(\bF)\cong\T_X^+\otimes\rC^*(\bF)$ completely isometrically for every (equivalently, some) discrete non-abelian free group $\bF$.
    \end{enumerate}
\end{proposition}

\begin{proof}
    {\rm (i)}$\Rightarrow${\rm (ii)}: Fix some discrete non-abelian free group $\bF$. Since $\A$ has the WEP, we have that $\A\otimes_{\max}\rC^*(\bF) = \A\otimes\rC^*(\bF)$ by \cite[Proposition 1.1]{kirchberg1993non}. For each $t\in X\odot\rC^*(\bF)$, we have $\|t\|^2 = \|\langle t,t\rangle\|$ where $\langle t, t\rangle\in \A\odot\rC^*(\bF)$. It follows that the norm of $t$ in $X\otimes_{\max}\rC^*(\bF)$ coincides with its norm in $X\otimes\rC^*(\bF)$. In other words, the canonical surjection $U:X\otimes_{\max}\rC^*(\bF)\rightarrow X\otimes\rC^*(\bF)$ fixing $X\odot\rC^*(\bF)$ is isometric. Thus, applying Lemma \ref{l:tens-max-incl} (iii) and \cite[Proposition 2.10]{dor2020classification}, we have
    \[
        \T_X\otimes_{\max}\rC^*(\bF)\cong\T_{X\otimes_{\max}\rC^*(\bF)}\cong \T_{X\otimes\rC^*(\bF)}\cong\T_X\otimes\rC^*(\bF)
    \]
    by a generator-preserving $*$-isomorphism. So, $\T_X$ has the WEP.
    
    {\rm (ii)}$\Rightarrow${\rm (iii)}: Upon applying Lemma \ref{l:tens-max-incl} (iii) and (iv), we see that
    \[
        \T_X^+\otimes_{\max}\rC^*(\bF)\cong\T_{X\otimes_{\max}\rC^*(\bF)}^+\subset\T_{X\otimes_{\max}\rC^*(\bF)}\cong\T_X\otimes_{\max}\rC^*(\bF)\cong\T_X\otimes\rC^*(\bF)
    \] 
    where the final $*$-isomorphism holds since $\T_X$ has the WEP. Since the maps are generator-preserving, it follows that $\T_X^+\otimes_{\max}\rC^*(\bF)\cong\T_X^+\otimes\rC^*(\bF)$.

    {\rm (iii)}$\Rightarrow${\rm (i)}: This follows immediately from Lemma \ref{l:tens-max-incl} {\rm (ii)} and Kirchberg's tensorial characterization of the WEP \cite{kirchberg1993non}.
\end{proof}

Finally, we consider nuclearity for tensor algebras. An operator algebra $\A$ is said to be \emph{$\rC^*$-nuclear} if $\A\otimes_{\max} \B = \A\otimes \B$ for every $\rC^*$-algebra $\B$. If, instead, one stipulates the tensorial condition across all non-self-adjoint operator algebras rather than only $\rC^*$-algebras, then this forces $\A$ itself to be a $\rC^*$-algebra \cite[6.1]{blecher2004operator}. We remark that $\rC^*$-nuclearity is generally more restrictive than its self-adjoint counterpart. For one, the example of Harris and Kim \cite[Theorem 5.6]{harris2019crossed} is a matrix algebra that is a nuclearity detector for $\rC^*$-algebras. Nevertheless, our results provide us access to a large class of $\rC^*$-nuclear operator algebras.

\begin{theorem}\label{t:tens-nuclear}
    Let $(\A, X)$ be a $\rC^*$-correspondence. Then $\T_X^+$ is $\rC^*$-nuclear if and only if $\A$ is nuclear.
\end{theorem}

\begin{proof}
    If $\T_X^+$ is $\rC^*$-nuclear, then the conclusion follows from Lemma \ref{l:tens-max-incl} (ii). Conversely, suppose that $\A$ is nuclear. Then $\A$ is exact and so $\T_X^+$ is exact by Proposition \ref{p:tens-exact}. Furthermore, since $\A$ has the WEP, by Proposition \ref{p:tens-wep} we have $\T_X^+\otimes_{\max}\rC^*(\bF)\cong\T_X^+\otimes\rC^*(\bF)$ for every discrete free group $\bF$. Upon noting that $\T_X^+$ has an approximate unit coming from $\A$, the conclusion now follows from \cite[Proposition 5.7]{blecher2011nuclearity}.
\end{proof}

For more examples of $\rC^*$-nuclear operator algebras, one may consult \cite[Section 6]{blecher2011nuclearity} or \cite[Theorems 3.2 and 3.4]{paulsen1990tensor}. However, Theorem \ref{t:tens-nuclear} manages to actually recover many of the $\rC^*$-nuclear operator algebras that have been recorded.

We remark that nuclearity properties for non-self-adjoint algebras have been studied since the work of Paulsen and Power \cite{paulsen1990tensor}, where it was seen to be related to obtaining a commutant lifting theorem (see \cite[Proposition 2.5]{paulsen1990tensor}).

\section{The Hao--Ng isomorphism theorem for full crossed products}\label{s:Hao--Ng}

In this section, we offer several different structural obstructions to the Hao--Ng isomorphism for full crossed products. That is, we determine the general non-existence of a generator-preserving $*$-isomorphism
\[
    \O_X \rtimes_\alpha G \cong \O_{X\rtimes_\alpha G}.
\]
We start by describing Katsoulis and Ramsey's reduction of the problem \cite{katsoulis2021non}, and the broader role this question plays in non-self-adjoint operator algebra theory.

For this, let $(\A, G, \alpha)$ be a dynamical system. From the standpoint of non-self-adjoint operator algebra theory, Katsoulis and Ramsey's work reduces the Hao--Ng isomorphism for full crossed products to two separate problems when $(\A, G, \alpha)$ is a tensor algebra dynamical system. First, the validity of the commutation of the full crossed product of the $\rC^*$-envelope with its canonical subalgebra:
\begin{equation}\label{eq:comm}
    \rC^*_e(\A) \rtimes_\alpha G \cong \rC^*_e(\A \rtimes_{\rC^*_e(\A), \varepsilon, \alpha}G).
\end{equation}
Second, the independence of the full crossed product construction:
\begin{equation}\label{eq:ind}
    \A \rtimes_{\rC^*_e(\A), \varepsilon, \alpha} G \cong \A\rtimes_\alpha G.
\end{equation}
For the Hao--Ng isomorphism for full crossed products to be valid, Katsoulis and Ramsey proved that both statements are necessary when $\A = \T_X^+$ and the action $\alpha$ arises from a $\rC^*$-correspondence dynamical system \cite[Theorem 4.9]{katsoulis2021non}. Indeed, equations (\ref{eq:comm}) and (\ref{eq:ind}), respectively, correspond to the existence of generator-preserving $*$-isomorphisms
\[
    \O_X\rtimes_\alpha G\cong \O_{X\hat{\rtimes}_\alpha G} \qquad \text{and} \qquad \O_{X\hat{\rtimes}_\alpha G} \cong \O_{X\rtimes_\alpha G}.
\]
Further, for hyperrigid $\rC^*$-correspondences, Katsoulis and Ramsey proved that (\ref{eq:ind}) alone is equivalent to the Hao--Ng isomorphism for full crossed products.

In Subsection \ref{ss:O-inf}, we present a failure of (\ref{eq:comm}) for the tensor algebra arising from the $\rC^*$-correspondence $\ell^2(\mathbb{N})$ over $\mathbb{C}$, which yields the Cuntz $\rC^*$-algebra $\O_\infty$, by constructing an appropriate action for an \emph{arbitrary} non-amenable countable discrete group. More broadly, this is also the first example of a non-self-adjoint dynamical system that violates condition (\ref{eq:comm}), thereby resolving a question of Katsoulis.

Accordingly, in Subsection \ref{ss:hr-Hao--Ng}, we concern ourselves with the failure of (\ref{eq:ind}) for hyperrigid $\rC^*$-correspondences. Harris and Kim witnessed a strong failure of property (\ref{eq:ind}) for arbitrary operator algebra dynamical systems \cite[Theorem 5.6]{harris2019crossed}. We show that this failure can be witnessed among tensor algebra dynamical systems and uncover precise structural obstructions to (\ref{eq:ind}), both at the level of the $\rC^*$-correspondence as well as at the $\rC^*$-correspondence dynamical system.

\subsection{An obstruction for non-hyperrigid C*-correspondences}\label{ss:O-inf}

Here, we construct an example of an operator algebra dynamical system $(\A, G, \alpha)$ for which there is no generator-preserving $*$-isomorphism
\[
    \rC^*_e(\A) \rtimes_\alpha G  \cong \rC^*_e(\A\rtimes_{\rC^*_e(\A),\varepsilon,\alpha} G).
\]

To this end, recall that the $\rC^*$-correspondence $X =\ell^2(\bN)$ over $\A = \bC$ with the trivial left action, gives rise to the $\rC^*$-correspondence that yields Cuntz $\rC^*$-algebra $\O_X\cong \O_\infty$ via a canonical generator preserving isomorphism. In this case, Katsura's ideal $J_X = 0$ acts degenerately on $X$. Moreover, by \cite[Corollary 4.8]{katsoulis2021non}, we remark that the canonical quotient map $\T_X^+\rtimes_\alpha G \rightarrow \T_X^+ \rtimes_{\O_\infty, \varepsilon,\alpha}G$ is completely isometric for any choice of generalized gauge action $\alpha:G\acts(\A,X)$. In which case, we have that $\O_{X\rtimes_\alpha G} \cong \O_{X\hat{\rtimes}_\alpha G}$ by a generator-preserving $*$-isomorphism and, in particular, there exists a surjective $*$-homomorphism $\O_X \rtimes_\alpha G \rightarrow \O_{X \rtimes_\alpha G}$. Nevertheless, there is an action by arbitrary countable non-amenable discrete groups for which the Hao--Ng isomorphism for full crossed products fails.

\begin{theorem}\label{t:infinite-cuntz}
    Let $\A = \mathbb{C}$ and $X = \ell^2(\mathbb{N})$, so that $\mathcal{O}_X \cong \mathcal{O}_{\infty}$ canonically. For every countable non-amenable discrete group $G$, there is a generalized gauge action $\alpha:G\acts(\A, X)$ for which the induced surjective $*$-homomorphism $\O_X \rtimes_\alpha G \rightarrow \O_{X\rtimes_\alpha G}$ is not injective.
\end{theorem}

\begin{proof}
    By \cite[Theorem 3.1]{bedos2015new} and on account of the fact that $\O_X\cong\T_X$, it follows that there is a generator-preserving $*$-isomorphism $\T_{X\rtimes_\alpha G}\cong \O_X \rtimes_\alpha G$ for any choice of generalized gauge action $\alpha:G\acts(\bC,X)$. Therefore, the conclusion then reduces to showing that the canonical quotient map $Q$ from $\T_{X\rtimes_\alpha G}$ onto $\O_{X\rtimes_\alpha G}$ cannot be faithful. By \cite[Proposition 6.5]{katsura2004c}, we have that
    \[
        \ker Q = \K(\F_{X\rtimes_\alpha G} J_{X\rtimes_\alpha G}).
    \]
    Thus, it suffices to construct a generalized gauge action $\alpha:G\acts(\bC, X)$ for which the Katsura ideal $J_{X\rtimes_\alpha G}$ of $X\rtimes_\alpha G$ is non-zero.

    For this, we regard $X \cong \bC\xi_0 \oplus\ell^2(G)$ as right Hilbert $\bC$-modules where $\xi_0$ is some fixed unit vector. Then, we let $\{T_\xi : \xi\in X\}$ denote the generators of $\O_\infty$ arising from the $\rC^*$-correspondence $(\bC, X)$. In addition, let $1_G$ denote the unique extension of the trivial representation of $G$ to a $*$-homomorphism of $\rC^*(G)$, and $\{u_g: g\in G\}$ denote the canonical unitaries in $\rC^*(G)$.
    
    Now, consider the unitary representation
    \[
        U = 1_G \oplus \lambda: G\rightarrow \bB(\bC\xi_0 \oplus \ell^2(G)),
    \]
    and define a generalized gauge action $\alpha: G\acts(\bC, X)$ by $\alpha_g(\xi) = U_g\xi$ for every $\xi\in X$. This induces an action on $\O_X$ that is determined by $\alpha_g(T_\xi) = T_{U_g\xi}$ and, in particular, we have that $U_g\xi_0 = \xi_0$. As $\rC^*$-correspondences over $\rC^*(G)$, we have
    \[
        X\rtimes_\alpha G\cong (\bC\xi_0 \oplus \ell^2(G)) \otimes \rC^*(G)
    \]
    where the left action is given by
    \[
        \varphi: \rC^*(G)\rightarrow \L( (\bC\xi_0 \oplus \ell^2(G)) \otimes \rC^*(G)), \qquad u_g \mapsto U_g\otimes u_g.
    \]
    Indeed, it is easily verified that the map
    \[
        W: \rC_c(G, \bC\xi_0 \oplus\ell^2(G)) \rightarrow (\bC\xi_0 \oplus \ell^2(G)) \odot \rC^*(G), \quad f\mapsto \sum_{g\in G} f(g)\otimes u_g,
    \]
    extends to a surjective, isometric map  $W: X\rtimes_\alpha G \rightarrow (\bC\xi_0 \oplus \ell^2(G)) \otimes \rC^*(G)$ that preserves the left action and right Hilbert $\rC^*(G)$-module structure on $X \rtimes_\alpha G$.
    
    As $G$ is non-amenable, there is a positive operator $a\in\rC^*(G)$ such that $1_G(a)=1$ and $\lambda(a) = 0$. We claim that $a$ lies in $J_{X\rtimes_\alpha G}$. For this, note that $\varphi$ is faithful and so, $a\in(\ker\varphi)^{\perp} = \rC^*(G)$. To see that $\varphi(a)$ lies in $\K((\bC\xi_0 \oplus \ell^2(G)) \otimes \rC^*(G))$, observe that
    \[
        (\bC\xi_0 \oplus\ell^2(G))\otimes\rC^*(G) \cong (\bC\xi_0 \otimes \rC^*(G)) \oplus (\ell^2(G) \otimes \rC^*(G))
    \]
    as right Hilbert $\rC^*(G)$-modules. With respect to this decomposition, we have that
    \[
        (U\otimes u)(a) = (1_G\otimes u)(a)\oplus (\lambda\otimes u)(a).
    \]
    On the second summand, the unitary representation $\lambda \otimes u$ is unitarily equivalent to $\lambda \otimes 1$ by Fell's absorption principle \cite[Theorem 2.5.5]{brown2008textrm} and thus, $(\lambda\otimes u)(a) = 0$ since $\lambda\otimes u$ factors through $\rC^*_\lambda(G)$. On the summand $\bC\xi_0 \otimes \rC^*(G)\cong\rC^*(G)$, we have that $\varphi(a)$ acts as left multiplication by $a$ since $1_G(a) = 1$. Therefore, this establishes that
    \[
        \varphi(a) = \theta_{\xi_0\otimes a^{1/2}, \xi_0 \otimes a^{1/2}} \in\K((\bC\xi_0 \oplus \ell^2(G)) \otimes \rC^*(G)).
    \]
    In turn, we may conclude that $a\in J_{X\rtimes_\alpha G}$ as claimed and that $J_{X\rtimes_\alpha G}$ is non-zero. It then follows that $\O_X\rtimes_\alpha G\rightarrow\O_{X\rtimes_\alpha G}$ fails to be injective.
\end{proof}

The key technical property that we require in Theorem \ref{t:infinite-cuntz} is that Katsura's ideal $J_X$ vanishes. In particular, by Kim's theorem \cite[Theorem 1.1]{kim2021hyperrigidity}, our construction must arise from a non-hyperrigid $\rC^*$-correspondence. Nevertheless, despite the technicalities that are present in the above proof, we remark that this strategy can be adapted to broader families of non-hyperrigid $\rC^*$-correspondences. For instance, the works of Deaconu offer natural families of candidates built from group actions on graphs \cites{deaconu2018group} and unitary representations of groups \cite{deaconu2018cuntz}.

\begin{corollary}
    Let $G$ be a non-amenable countable discrete group. Then, there is a dynamical system $(\A, G, \alpha)$ such that the surjective $*$-homomorphism $\rC^*_e(\A) \rtimes_\alpha G \rightarrow \rC^*_e(\A \rtimes_{\rC^*_e(\A), \varepsilon,\alpha}G)$ is not injective.
\end{corollary}

\begin{proof}
    Let $\alpha:G\acts(\bC, \ell^2(\bN))$ denote the generalized gauge action given by Theorem \ref{t:infinite-cuntz}. We claim that the induced tensor algebra dynamical system $(\T_X^+, G, \alpha)$ satisfies the desired conclusion. Indeed, since $J_X=0$, we have that
    \[
        \T_X^+ \rtimes_{\O_\infty, \varepsilon,\alpha}G \cong \T_X^+ \rtimes_\alpha G \cong \T_{X\rtimes_\alpha G}^+
    \]
    by \cite[Theorem 4.7]{katsoulis2021non}. Therefore, by \cite{katsoulis2006tensor}, it follows that
    \[
        \rC^*_e(\T_X^+ \rtimes_{\O_X, \varepsilon,\alpha}G)\cong \rC^*_e(\T_{X\rtimes_\alpha G}^+)\cong\O_{X\rtimes_\alpha G}.
    \]
    On the other hand, by \cite{katsoulis2006tensor}, we have that $\rC^*_e(\T_X^+)\rtimes_\alpha G\cong \O_X\rtimes_\alpha G$. Therefore, the conclusion follows immediately from Theorem \ref{t:infinite-cuntz}.
\end{proof}

\subsection{Obstructions for hyperrigid C*-correspondences}\label{ss:hr-Hao--Ng}

In this subsection, we uncover failures of the Hao--Ng isomorphism for full crossed products of hyperrigid $\rC^*$-correspondences. To accomplish this, we invoke a strategy that is reminiscent of Harris and Kim's argument \cite[Theorem 5.6]{harris2019crossed}. In this direction, we first uncover constraints at the level of the $\rC^*$-correspondence for the validity of the Hao--Ng isomorphism for full crossed products against the trivial action (Theorem \ref{t:coeff-obs}). We then present an obstruction at the level of the coefficient algebra of the $\rC^*$-correspondence dynamical system (Theorem \ref{t:act-obs}), and conclude by constructing a class of examples where either of the given criteria are met (Examples \ref{e:tens} and \ref{e:top-graph}). 

We now extract our obstruction to the Hao--Ng isomorphism for trivial actions by hyperrigid $\rC^*$-correspondences. This highlights a restriction on the validity of the full Hao--Ng isomorphism at the level of the coefficient $\rC^*$-algebra of the $\rC^*$-correspondence. We highlight to the reader that, for hyperrigid $\rC^*$-correspondences, there always exists a generator-preserving surjective $*$-homomor\-phism $ \O_{X\rtimes_\alpha G} \rightarrow \O_X \rtimes_\alpha G$ on account of \cite[Theorem 3.6]{katsoulis2021non}.

\begin{theorem}\label{t:coeff-obs}
    Let $(\A, X)$ be a hyperrigid $\rC^*$-correspondence where $\O_X$ has the WEP, yet $\A$ does not have WEP. Let $\bF$ be a discrete non-abelian free group. Then, the canonical surjective $*$-homomorphism $ \O_{X\rtimes_{\id} \bF} \rightarrow\O_X\rtimes_{\id}\bF$ is not injective.
\end{theorem}

\begin{proof}
    Since $(\A, X)$ is hyperrigid, by \cite[Theorem 4.9]{katsoulis2021non}, it suffices to prove that the canonical quotient map $\T_X^+ \rtimes_{\id} \bF \rightarrow \T_X^+ \rtimes_{\O_X, \varepsilon,\id}\bF$ fails to be completely isometric. By Proposition \ref{p:tens-wep}, the canonical surjection $\T_X^+\otimes_{\max}\rC^*(\bF) \rightarrow \T_X^+\otimes \rC^*(\bF)$ is not completely isometric. Now, note that $\T_X^+ \rtimes_{\id} \bF$ is completely isometrically isomorphic to $\T_X^+\otimes_{\max}\rC^*(\bF)$. Indeed, we have
    \[
        \rC^*_{\max}(\T_X^+)\rtimes_{\id} \bF \cong \rC^*_{\max}(\T_X^+) \otimes_{\max}\rC^*(\bF)
    \]
    since these $\rC^*$-algebras possess the same universal property. Thus, the restriction
    \[
        \Phi:\T_X^+\rtimes_{\id} \bF \rightarrow \T_X^+\otimes_{\max} \rC^*(\bF), \quad au_w\mapsto a\otimes u_w
    \]
    determines a completely isometric isomorphism.

    On the other hand, we claim that $\T_X^+\rtimes_{\O_X,\varepsilon,\id}\bF$ is completely isometrically isomorphic to $\T_X^+\otimes \rC^*(\bF)$. Indeed, since $\O_X$ has WEP, we have that 
    \[
        \O_X\otimes\rC^*(\bF)\cong\O_X\otimes_{\max}\rC^*(\bF) = \O_X\rtimes_{\id}\bF
    \]
    where the $*$-isomorphisms are the canonical ones. By injectivity of minimal tensor products, this then guarantees us that $\T_X^+\rtimes_{\O_X,\varepsilon,\id}\bF$ is completely isometrically isomorphic to $\T_X^+\otimes \rC^*(\bF)$, as desired. Consequently, since
    \[
        \T_X^+\rtimes_{\id}\bF \cong \T_X^+\otimes_{\max}\rC^*(\bF), \ \ \T_X^+\rtimes_{\O_X,\varepsilon,\id}\bF\cong \T_X^+\otimes\rC^*(\bF) 
    \]
    via generator-preserving isomorphisms, we conclude that the canonical completely contractive map $\T_X^+\rtimes_{\id}\bF \rightarrow \T_X^+\rtimes_{\O_X,\varepsilon,\id}\bF$ is not completely isometric.
\end{proof}

Our next goal is to reveal another hurdle to the full Hao--Ng isomorphism for hyperrigid $\rC^*$-correspondences. In accordance with Theorem \ref{t:coeff-obs}, this gives rise to a further obstruction to the isomorphism in terms of the coefficient $\rC^*$-algebra that arises from the $\rC^*$-correspondence dynamical system.

\begin{theorem}\label{t:act-obs}
    Let $X$ be a hyperrigid $\rC^*$-correspondence over a $\rC^*$-algebra $\A$, and let $\alpha : G \acts (\A,X)$ be a generalized gauge action. Then, the following hold.
    \begin{enumerate}[{\rm (i)}]
        \item If the canonical map $\A\rtimes_\alpha G\rightarrow \A \rtimes_r G$ is injective, then $\O_{X\rtimes_\alpha G} \cong \O_X \rtimes_\alpha G$.
        \item Assume that $\O_X\rtimes_\alpha G \rightarrow\O_X \rtimes_r G$ is injective, yet $\A\rtimes_\alpha G\rightarrow \A \rtimes_r G$ is not. Then, the canonical surjective $*$-homomorphism $\O_{X\rtimes_\alpha G} \rightarrow \O_X\rtimes_\alpha G$ is not injective.
    \end{enumerate}
\end{theorem}

\begin{proof}
    (i): By \cite[Theorem 3.6]{bedos2015new}, the assumption implies that the canonical maps
    \[
        X \rtimes_\alpha G \longrightarrow X\rtimes_{\alpha, r}G \qquad \text{and} \qquad \T_X \rtimes_\alpha G \longrightarrow \T_X\rtimes_{\alpha, r}G
    \]
    are isometric. Upon combining this with \cite[Theorem 4.7]{katsoulis2021non}, we then have that
    \[
        \T_X^+ \rtimes_{\alpha, r} G \cong \T_X^+ \rtimes_{\T_X, \id, \alpha} G \cong \T_X^+ \rtimes_\alpha G \cong  \T_{X \rtimes_\alpha G}^+ \cong \T_{X\rtimes_{\alpha, r}G}^+.
    \]
    By \cite[Corollary 3.16]{katsoulis2019crossed}, there is a generator-preserving completely contractive map
    \[
        \sigma:\T_X^+ \rtimes_{\O_X,\varepsilon,\alpha}G \rightarrow \T_X^+ \rtimes_{\alpha, r} G.
    \]
    As $\T_X^+\rtimes_{\alpha,r}G\cong \T_X^+ \rtimes_{\T_X,\id,\alpha}G$, it follows that $\sigma$ has a completely contractive inverse and thus, is completely isometric. Therefore, the canonical map $\T_X^+ \rtimes_\alpha G \rightarrow \T_X^+ \rtimes_{\O_X, \varepsilon, \alpha} G$ is completely isometric. As $(\A, X)$ is hyperrigid, the conclusion now follows by \cite[Theorem 4.9]{katsoulis2021non}.
    
    (ii): As $(\A, X)$ is hyperrigid, it is sufficient to prove that the canonical quotient map $\T_X^+ \rtimes_\alpha G \rightarrow \T_X^+ \rtimes_{\O_X, \varepsilon, \alpha} G$ is not completely isometric by \cite[Theorem 4.9]{katsoulis2021non}. To this end, first note that $\T_X^+ \rtimes_{\alpha,r} G$ is completely isometrically contained in $\O_X \rtimes_{\alpha, r}G$ by \cite[Corollary 3.16]{katsoulis2019crossed}. Then, as the canonical map $\O_X \rtimes_\alpha G \rightarrow \O_X \rtimes_{\alpha, r}G$ is injective, it follows that $\T_X^+ \rtimes_{\O_X, \varepsilon,\alpha} G \cong \T_X^+ \rtimes_{\alpha, r}G$ completely isometrically.
    
    Now, by \cite[Theorem 4.7]{katsoulis2021non}, we have that the canonical $*$-homomorphism $\lambda: A\rtimes_\alpha G \rightarrow \T_X^+ \rtimes_\alpha G$ is injective. By injectivity of the reduced crossed product, there is then a commuting diagram of completely contractive homomorphisms
    \[\begin{tikzcd}
        \A\rtimes_\alpha G \arrow[r]  \arrow[d, "\lambda"] & \A \rtimes_{\alpha, r} G \arrow[d]\\
        \T_X^+ \rtimes_\alpha G \arrow[r] & \T_X^+ \rtimes_{\alpha, r} G
    \end{tikzcd}\]
    As the canonical quotient map $\A\rtimes_\alpha G \rightarrow \A\rtimes_r G$ fails to be injective, it follows that $\T_X^+ \rtimes_\alpha G \rightarrow \T_X^+ \rtimes_{\O_X,\varepsilon,\alpha} G$ is not injective either.
\end{proof}

Our final goal is to construct examples where the conditions of Theorem \ref{t:coeff-obs} and Theorem \ref{t:act-obs} (ii) are met. We remark that contemporaneous work of Mukohara and Suzuki \cite[Theorem 7.4]{mukohara2026stabilization} gives a family of $\rC^*$-correspondence dynamical systems satisfying Theorem \ref{t:act-obs} (ii). Our own counterexample shows that this failure occurs even among topological graphs. To this end, we recall a construction due to Katsura \cite[Example 7.7]{katsura2004c}, which is based on an idea communicated by Ozawa.

Given an inclusion $\C\subset \B$ of $\rC^*$-algebras, we construct a $\rC^*$-correspondence that we will denote by $(\A_{\B,\C}, X_{\B,\C})$. First, set $\A_{\B,\C} = \bigoplus_{n\in\bZ}\A_{\B,\C}(n)$ where
\[
    \A_{\B,\C}(n) =
    \begin{cases}
        \B, & n\geq 1\\
        \C, & n\leq 0.
    \end{cases}
\]
Let $\varphi:\A_{\B,\C}\rightarrow \A_{\B,\C}$ be an injective $*$-homomorphism where $\varphi|_{\A_{\B,\C}(0)}: \A_{\B,\C}(0)\rightarrow \A_{\B,\C}(1)$ is the inclusion map and, for each non-zero $n\in\bZ$, we have $\varphi|_{\A_{\B,\C}(n)}: \A_{\B,\C}(n)\rightarrow \A_{\B,\C}(n+1)$ is the identity map. We then define a $\rC^*$-correspondence $(\A_{\B,\C}, X_{\B,\C})$ with the action
\[
    \varphi_{X_{\B,\C}}: \A_{\B,\C}\rightarrow\L(X_{\B,\C})
\]
being given by the composition of $\varphi$ and the isomorphism $\A_{\B,\C}\cong\K(X_{\B,\C})$, where $X_{\B,\C} = \A_{\B,\C}$ is equipped with the natural right Hilbert $\A_{\B,\C}$-module structure.

We make a few remarks on this construction. First note that the left action $\varphi_{X_{\B,\C}}$ is injective and maps into $\K(X_{\B,\C})$ and so, $(\A_{\B,\C}, X_{\B,\C})$ is regular. In particular, by \cite[Theorem 1.1]{kim2021hyperrigidity}, we have that $(\A_{\B,\C}, X_{\B,\C})$ is a hyperrigid $\rC^*$-correspondence. Further, suppose that $\gamma:\bT\acts\O_{X_{\B,\C}}$ is the canonical gauge action and $\O_{X_{\B,\C}}^\gamma$ denotes the fixed point algebra. Then, by \cite[Proposition 5.7]{katsura2004c}, we have that
\[
    \O_{X_{\B,\C}}^\gamma \cong \varinjlim(\A_{\B,\C}, \varphi) \cong \bigoplus_{n\in\bZ}\B.
\]
We now leverage these facts to construct examples where the obstructions of Theorem \ref{t:coeff-obs} and \ref{t:act-obs} are witnessed.

\begin{example}\label{e:tens}
    Consider an inclusion of $\rC^*$-algebras $\C\subset \B$ where $\B$ has the weak expectation property, yet $\C$ does not. For example, $\rC^*_\lambda(\bF_2)$ does not have the WEP by \cite[Proposition 3.6.9]{brown2008textrm}, yet $B(\ell^2(\bF_2))$ has the WEP by \cite[Remark 15.2 (ii)]{pisier2003introduction}. Then, the $\rC^*$-correspondence $(\A, X) = (\A_{\B,\C}, X_{\B,\C})$ satisfies the constraints of Theorem \ref{t:coeff-obs}. Indeed, note that the fixed point algebra of the gauge action
    \[
        \O_X^\gamma \cong \bigoplus_{n\in\bZ} \B
    \]
    has the WEP, since it is a direct sum of $\rC^*$-algebras with WEP. In turn, it follows from the proof of \cite[Theorem 4.5.2]{brown2008textrm} (by replacing $B$ with $C^*(\mathbb{F})$ in the proof) that $\O_X$ has the WEP as well. However, note that $\A$ cannot have the WEP as $\C$ is a direct summand of $\A$ that fails to have the WEP. Therefore, by Theorem \ref{t:coeff-obs}, the canonical map $\O_{X\rtimes_{\id} \bF} \rightarrow \O_X\rtimes_{\id}\bF$ fails to be injective for every non-abelian free group $\bF$.
\end{example}

Finally, we construct examples that satisfy Theorem \ref{t:act-obs} (ii). For this, recall that a finitely generated group $G$ is \emph{hyperbolic} if its Cayley graph is a hyperbolic metric space. Note that free groups are hyperbolic \cite[Remark 5.3.7]{brown2008textrm}. While we do not require precise details, the reader may consult \cite[Section 5.3]{brown2008textrm} for background.

\begin{example}\label{e:top-graph}
    Let $G$ be a non-amenable hyperbolic group and $\partial G$ denote its Gromov boundary \cite[Definition 5.3.9]{brown2008textrm}. Recall that the action $\beta:G\acts\rC(\partial G)$ induced by left translation is amenable \cite[Theorem 5.3.15]{brown2008textrm}. We will construct a $\rC^*$-correspondence dynamical system from
    \[
        (\A, X) = (\A_{\rC(\partial G),\bC}, X_{\rC(\partial G),\bC}).
    \]
    
    Define a generalized gauge action $\alpha: G\acts(\A, X)$ by
    \[
        \alpha_g((a_n)_{n \in \mathbb{Z}}) = (\ldots, a_{-1}, a_0, \beta_g(a_1), \beta_g(a_2),\ldots), \ \ \text{for} \ \ (a_n)_{n\in\bZ}\in\A,\ g\in G,
    \]
    and which is defined the same on $X$. Observe that the action $\alpha: G\acts\A$ is non-amenable. Indeed, a one-dimensional direct summand of $\A$ is $\alpha$-invariant and, as $G$ is non-amenable, the restricted action on such a summand is non-amenable. In turn, the action $\alpha:G\acts\A$ must therefore be non-amenable \cite[Proposition 3.23]{buss2024amenability}. As $\A$ is commutative, it then follows that $\A\rtimes_\alpha G \rightarrow\A\rtimes_{\alpha, r} G$ cannot be injective \cite[Theorems 4.3.4 and 4.4.3]{brown2008textrm}.

    On the other hand, we claim that the induced action $\alpha:G\acts\O_X$ is amenable. Indeed, let $\gamma:\bT\acts\O_X$ denote the canonical gauge action and $\O_X^\gamma$ denote the fixed point algebra. Note that there is a well-defined action $\alpha\times\gamma: G\times\bT\acts\O_X$ as $\alpha$ and $\gamma$ commute. Since the action $G\acts\rC(\partial G)$ is amenable, it follows that the induced action $\alpha:G\acts\O_X^\gamma \cong \bigoplus_{n\in \mathbb{Z}} \rC(\partial G)$ is then amenable as well. Thus, by \cite[Theorem 5.1]{ozawa2021characterizations}, we deduce that the action $\alpha:G\acts\O_X$ is also amenable.
    
     By \cite[Theorem 4.3.4]{brown2008textrm}, it follows that the canonical $*$-homomorphism $\O_X\rtimes_\alpha G \rightarrow\O_X\rtimes_{\alpha, r} G$ is injective. In turn, by Theorem \ref{t:act-obs}, it follows that the quotient map $\O_{X\rtimes_\alpha G}\rightarrow \O_X\rtimes_\alpha G$ is not injective.
\end{example}

We remark that Example \ref{e:top-graph} also works for $(\A_{\rC(X),\bC}, X_{\rC(X), \bC})$ whenever there is an amenable action $\alpha:G\acts \rC(X)$ by a non-amenable group $G$. Furthermore, examples of this form demonstrate that the Hao--Ng isomorphism for full crossed products fails even for regular topological graph $\rC^*$-correspondences.

To see this, recall that a \emph{topological graph} is a tuple $(E^0, E^1, r, s)$ where $E^0, E^1$ are locally compact topological spaces, where the source and range maps $s,r: E^1\rightarrow E^0$ are continuous, and $s$ is a local homeomorphism. There is a natural $\rC_0(E^0)$-valued inner product on $\rC_c(E^1)$ given by
\[
    \langle \xi, \eta\rangle(v) = \sum_{e\in s^{-1}(v)} \ol{\xi(e)} \eta(e), \ \ \text{for} \ \  \xi,\eta\in\rC_c(E^1),
\]
and the completion $X(E)$ of $\rC_c(E^1)$ with respect to this inner product gives a Hilbert $\rC^*$-module. There is also a left action $\varphi$ of $\rC_0(E^0)$ on $X$ satisfying $\varphi(f)\xi(t) = f(r(t))\xi(t)$ for every $\xi\in\rC_c(E^1)$, which gives $(\rC_0(E^0), X(E))$ the structure of a $\rC^*$-correspondence. 

The $\rC^*$-correspondence $(\A_{\rC(X),\bC}, X_{\rC(X), \bC})$ then arises from a topological graph by setting $E^0$ and $E^1$ to be the Gelfand spectra of $\bigoplus_{n\in \mathbb{Z}} \A_{\rC(X), \bC}(n)$. Then, one takes the source map $s:E^1\rightarrow E^0$ to be the identity map and the range map $r:E^1\rightarrow E^0$ to be the continuous map dual to $\varphi_{X_{\rC(X),\bC}}$ under Gelfand duality.

\bibliographystyle{plain}
\bibliography{bbl}

\end{document}